\documentclass[12pt,4paper]{amsart}
\usepackage{amssymb}
\usepackage{amsmath}
\usepackage{pdfsync}
\usepackage{verbatim}
\usepackage{datetime}
\usepackage{times}
\allowdisplaybreaks[2]
\makeatletter
\let\Thebibliography=\thebibliography
\renewcommand{\thebibliography}[1]{\def\@mkboth##1##2{}\Thebibliography{#1}
\addcontentsline{toc}{section}{References}
\frenchspacing % Maybe not needed
\setlength{\@topsep}{0pt}% Delete if extra space before list
\setlength{\itemsep}{0pt}%
\setlength{\parskip}{0pt plus 2pt}%
\makeatother
\makeatletter
\let\Enumerate=\enumerate
\renewcommand{\enumerate}{\Enumerate}%
\setlength{\@topsep}{0pt}% Delete if extra space before list
\setlength{\itemsep}{0pt}%
\setlength{\parskip}{0pt plus 1pt}%
\renewcommand{\theenumi}{\textup{(\alph{enumi})}}%
\renewcommand{\labelenumi}{\theenumi}%
}
\let\endEnumerate=\endenumerate
\renewcommand{\endenumerate}{\endEnumerate\unskip}
\makeatother
\makeatletter
\def\@seccntformat#1{\csname the#1\endcsname.\quad}
\makeatother
\newcommand{\art}[6]{{\sc #1, \rm #2, \it #3 \bf #4 \rm (#5), \mbox{#6}.}}
\newcommand{\book}[3]{{\sc #1, \it #2, \rm #3.}}
\newcommand{\AND}{{\rm and }}
\RequirePackage{amsthm}
\newtheoremstyle{descriptive}%
  {\topsep}   %{\medskipamount}          % Space above
  {\topsep}   %  {\medskipamount}          % Space below
  {\rmfamily} % Body font
  {}          % Indent
  {\bfseries} % Head font
  {.}         % Punctuation after thm head
  { }         % Space after thm head
  {}          % Thm head spec(?)
\newtheoremstyle{propositional}%
  {\topsep}   %  {\medskipamount}          % Space above
  {\topsep}   %  {\medskipamount}          % Space below
  {\itshape}  % Body font
  {}          % Indent
  {\bfseries} % Head font
  {.}         % Punctuation after thm head
  { }         % Space after thm head                      
  {}          % Thm head spec(?)
\theoremstyle{propositional}
\newtheorem{thm}{Theorem}[section]
\newtheorem{theorem}[thm]{Theorem}  % As used by Niko
\newtheorem{prop}[thm]{Proposition}

\newtheorem{lemma}[thm]{Lemma} % As used by Niko

\theoremstyle{descriptive}

\newtheorem{remark}[thm]{Remark}
\makeatletter
\renewenvironment{proof}[1][\proofname]{\par
  \pushQED{\qed}%
  \normalfont
  \trivlist
  \item[\hskip\labelsep
        \itshape
    #1\@addpunct{.}]\ignorespaces
}{%
  \popQED\endtrivlist\@endpefalse
}
\makeatother
\newdimen\extrawidth
\def\iintlim#1#2{\setbox0\hbox{$\scriptstyle#1$}%
        \setbox1\hbox{$\scriptstyle#2$}%
        \extrawidth=\wd1 \advance\extrawidth-\wd0
        \ifdim\extrawidth<0pt \extrawidth=0pt\fi%
        \int_{#1\kern\extrawidth \kern .5em}^{#2\kern -\wd1} \kern -.5em%
}
\def\vint_#1{\mathchoice%
          {\mathop{\kern 0.2em\vrule width 0.6em height 0.69678ex depth -0.58065ex
                  \kern -0.8em \intop}\nolimits_{\kern -0.4em#1}}%
          {\mathop{\kern 0.1em\vrule width 0.5em height 0.69678ex depth -0.60387ex
                  \kern -0.6em \intop}\nolimits_{#1}}%
          {\mathop{\kern 0.1em\vrule width 0.5em height 0.69678ex depth -0.60387ex
                  \kern -0.6em \intop}\nolimits_{#1}}%
          {\mathop{\kern 0.1em\vrule width 0.5em height 0.69678ex depth -0.60387ex
                  \kern -0.6em \intop}\nolimits_{#1}}}
\def\vintslides_#1{\mathchoice%
          {\mathop{\kern 0.1em\vrule width 0.5em height 0.697ex depth -0.581ex
                  \kern -0.6em \intop}\nolimits_{\kern -0.4em#1}}%
          {\mathop{\kern 0.1em\vrule width 0.3em height 0.697ex depth -0.604ex
                  \kern -0.4em \intop}\nolimits_{#1}}%
          {\mathop{\kern 0.1em\vrule width 0.3em height 0.697ex depth -0.604ex
                  \kern -0.4em \intop}\nolimits_{#1}}%
          {\mathop{\kern 0.1em\vrule width 0.3em height 0.697ex depth -0.604ex
                  \kern -0.4em \intop}\nolimits_{#1}}}
\newcommand{\Cp}{{C_p}}
\DeclareMathOperator{\diam}{\rm diam}

\DeclareMathOperator{\Lip}{Lip}

\DeclareMathOperator{\pr}{pr}

\DeclareMathOperator{\capc}{cap}
\DeclareMathOperator{\Capc}{Cap}

\newcommand{\loc}{_{\rm loc}}
\def\rightangle{\vcenter{\hsize5.5pt
    \hbox to5.5pt{\vrule height7pt\hfill}
    \hrule}}
\def\rtangle{\mathrel{\rightangle}}
\def\intave#1{\int_{#1}\hbox{\llap{$\raise2.3pt\hbox{\vrule
height.9pt width7pt}\phantom{\scriptstyle{#1}}\mkern-2mu$}}}
\newcommand{\dmu}{d\mu}

\newcommand{\eps}{\varepsilon}

\newcommand{\Om}{\Omega}

\renewcommand{\phi}{\varphi}
\newcommand{\p}{{$p\mspace{1mu}$}}
\newcommand{\R}{\mathbb{R}}

\newcommand{\limminus}{{\mathchoice{\raise.17ex\hbox{$\scriptstyle -$}}
                {\raise.17ex\hbox{$\scriptstyle -$}}
                {\raise.1ex\hbox{$\scriptscriptstyle -$}}
                {\scriptscriptstyle -}}}
\newcommand{\limplus}{{\mathchoice{\raise.17ex\hbox{$\scriptstyle +$}}
                {\raise.17ex\hbox{$\scriptstyle +$}}
                {\raise.1ex\hbox{$\scriptscriptstyle +$}}
                {\scriptscriptstyle +}}}
\newcommand{\limpm}{{\mathchoice{\raise.17ex\hbox{$\scriptstyle \pm$}}
                {\raise.17ex\hbox{$\scriptstyle \pm$}}
                {\raise.16ex\hbox{$\scriptscriptstyle \pm$}}
                {\scriptscriptstyle \pm}}}
\newcommand{\Np}{N^{1,p}}

\makeatletter
\newcommand{\setcurrentlabel}[1]{\def\@currentlabel{#1}}
\makeatother
\numberwithin{equation}{section}
\newcommand{\avint}{\vint}

 \newcommand{\GG}{\mathcal{G}}
\renewcommand{\H}{\mathcal{H}}

 \newcommand{\V}{{\mathcal V}}

\def\medcup{\mathop{\textstyle\bigcup}\limits}

\begin{document}

\title[Aspects of area formulas on metric measure spaces]{Aspects of
  area formulas by way of Luzin, Rad\'o, and Reichelderfer on metric
  measure spaces}

\author{Niko Marola} \address[Niko Marola]{Department of Mathematics
  and Statistics, University of Helsinki, P.O. Box 86, FI-00014
  University of Helsinki, Finland} \email{niko.marola@helsinki.fi}

\author{William P. Ziemer}
\address[William P. Ziemer]{Mathematics Department,
Indiana University, Bloomington, Indiana 47405, USA}
\email{ziemer@indiana.edu}

\date{}

\begin{abstract}
  We consider some measure-theoretic properties of functions belonging
  to a Sobolev-type class on metric measure spaces that admit a
  Poincar\'e inequality and are equipped with a doubling measure. The
  properties we have selected to study are those that are related to
  area formulas.
\end{abstract}

\maketitle



\section{Introduction}

We investigate some measure-theoretic properties of functions
belonging to the Banach or vector space-valued Newtonian space
$N^{1,p}(X)$ and compare these properties in the more general setting
with the classical Euclidean ones. Newtonian space is a metric space
analogue of the classical Sobolev space $W^{1,p}(\R^n)$ and was first
introduced and studied by Shanmugalingam in \cite{Sh-rev}; here $X$
refers to a complete metric measure space with a measure $\mu$ that
satisfies a volume doubling condition and the space is assumed to
support a Poincar\'e inequality. Under these rather standard
conditions on the space, we give a metric space version of Luzin's
condition for the graph mapping similar to one in Mal\'y et
al.~\cite{MaSwaZi}, we study absolute continuity as defined by
Mal\'y~\cite{Maly} for functions in the Newtonian class, and we also
discuss the condition due to Rad\'o and Reichelderfer~\cite{RaRe}.

We provide a version of the area formula for Newtonian functions. In
particular, we extend the Euclidean results of Haj\l
asz~\cite{HajlaszProc} and Mal\'y et al.~\cite{MaSwaZi} to
Newton--Sobolev functions in the aforementioned setting of general
metric spaces. We provide another view to a recent result by
Magnani~\cite{MagnaniX} which is related to the area formula in
general metric measure spaces.

Under rather general assumptions on $X$ (see Section~\ref{sect:prel})
the following area formula will be shown to be valid for the graph
mapping $\bar u$ of $u \in N\loc^{1,p}(X;\R^m)$, where $p > m$ or $p
\geq m=1$,
\[
 \H^Q(\bar u(A)) = \int_A \mathcal{J}\bar u\, d\mu,
\]  
whenever $A$ is a $\mu$-measurable subset and $\mathcal{J}\bar u$
denotes the generalized Jacobian of $\bar u$. In particular,
$\H^Q(\bar u(A))=0$ whenever $\mu(A)=0$. Here the exponent $Q$ serves as
a substitute for the dimension of $X$, and it is associated with the
doubling constant of the underlying measure $\mu$ (see
Section~\ref{sect:prel}).

Althought the proofs for these formulas are rather standard, our
general setting causes some unexpected difficulties. To overcome
these, we carefully consider some local properties of so-called
generalized Jacobian of a function and couple them with the
aforementioned measure-theoretic properties of Newton--Sobolev
functions. 

There is a rich supply of examples of complete metric spaces with a
volume doubling measure that support a Poincar\'e inequality and where
our results are applicable. To name but a few, we list
Carnot--Carath\'eodory spaces, thus including the Heisenberg group and
more general Carnot groups, as well as Riemannian manifolds with
non-negative Ricci curvature.

In outline, the paper is organized as follows: In
Section~\ref{sect:prel} we introduce the necessary background material
such as the doubling condition for the measure, upper gradients,
Poincar\'e inequality, Newtonian spaces, and capacity. In
Section~\ref{sect:graph} we establish a general criterion for a
version of Luzin's condition in the spirit of Rad\'o and
Reichelderfer~\cite[V.3.6]{RaRe}, see also Mal\'y et
al.~\cite{MaSwaZi}. Then we close Section~\ref{sect:graph} by proving,
with the aid of estimates between the capacity and the Hausdorff
content, that the graph mapping of a vector-valued Newtonian function
satisfies a version of the Luzin condition. In Section~\ref{sect:area}
we deal with the area formula. In Section~\ref{sect:RR} we study the
Rad\'o--Reichelderfer condition and absolute continuity of Newtonian
functions in the spirit of Mal\'y~\cite{Maly}.

\subsubsection*{Acknowledgements} We would like to thank Nageswari
Shanmugalingam for detailed comments and suggestions on several draft
versions of the paper.

\section{Metric measure spaces: doubling and Poincar\'e}
\label{sect:prel}

We briefly recall the basic definitions and collect some well-known
results needed later. For a thorough treatment we refer the reader to
a monograph by A. and J. Bj\"orn~\cite{BBbook} and
Heinonen~\cite{heinonen}.

Throughout the paper, if not otherwise stated, $X:=(X,d,\mu)$ is a
complete metric space endowed with a metric $d$ and a positive
complete Borel regular measure $\mu$ such that $0< \mu(B(x,r))
<\infty$ for all balls $B(x,r):=\{y\in X:\ d(x,y)<r\}$; and if
$B=B(x,r)$, then we denote $\tau B = B(x,\tau r)$ for each
$\tau>0$. We also denote the metric ball $B(x,r)$ by $B_X(x,r)$ if
necessary. Also throughout the paper, if not otherwise stated, let
$Y:=(Y,\tilde d,\nu)$ be a complete separable metric measure space with a
positive complete Borel regular measure $\nu$. A function
$f:X\to Y$ is called $L$-Lipschitz if for all $x,y\in X$, $\tilde
d(f(x),f(y))\leq Ld(x,y)$. We let $\Lip(f)$ be the infimum of such
$L$. 

In our treatment, it is natural to assume some connection between the
measure and the metric. Also by dimension we mean some quantity which
relates the measure of a metric ball to its radius. We shall clarify
these concepts below. Our standing assumptions on the metric space $X$
are as follows.

\begin{enumerate}

\item[(D)] The measure $\mu$ is doubling, i.e., there exists a
  constant $C_\mu \geq 1$, called the \emph{doubling constant} of
  $\mu$, such that
\begin{equation*}
        \mu(B(x,2r)) \le C_\mu \mu(B(x,r)).
\end{equation*}
for all $x\in X$ and $r>0$.

\item[(PI)] The space $X$ supports a weak $(1,p)$-Poincar\'e
  inequality for some $p\geq 1$ (see below).

\end{enumerate}

\medskip

We note the doubling condition (D) implies that for every $x\in X$
and $r>0$, we have for $\lambda \geq 1$
\begin{equation} \label{doublingcor}
\mu(B(x,\lambda r)) \leq C\lambda^Q\mu(B(x,r)),
\end{equation}
where $Q =\log_2C_\mu$, and the constant depends only on $C_\mu$. The
exponent $Q$ serves as a dimension of the doubling measure $\mu$; we
emphasize that it need not be an integer. \emph{When it is necessary
  to emphasize the relationship between $Q$ and $X$, we will use the
  notation $X^Q$}. Complete metric spaces verifying condition (D) are
precisely those that have finite Assouad dimension
\cite{heinonen}. This notion of dimension, however, need not to be
uniform in space. In what follows, we assume further that
there exists a constant $C >0$, depending only on $C_\mu$, such that
the measure $\mu$ satisfies the lower mass bound
\begin{equation} \label{measuregrowth}
Cr^Q \leq \mu(B(x,r))
\end{equation}
for all $x\in X$ and $0<r<\diam(X)$. It follows from (D) that $\mu$
satisfies the following local version of \eqref{measuregrowth}: For a
fixed $x_0\in X$ and a scale $r_D>0$ we have
\begin{equation} \label{localmassest}
\tilde{C}r^Q\leq \mu(B(x,r))
\end{equation}
for all balls $B(x,r)\subset X$ with $x\in B(x_0,r_D)$ and $0<r<r_D$,
where $\tilde{C} = Cr_D^{-Q}\mu(B(x_0,r_D))$ and $C$ is from
\eqref{doublingcor}.

Let $s\geq 0$. We define the (spherical) Hausdorff $s$-measure in $X$
as in Federer~\cite[2.10.2]{Federer} (see also \cite{heinonen}) and
will denote it by $\H^s$. We also denote by $\H^s_\infty$ the
Hausdorff $s$-content in $X$ defined as
\[
\H^s_\infty(E)=\inf\bigg\{\sum_{i=1}^\infty r_i^s:\
E\subset\bigcup_{i=1}^\infty B(x_i,r_i),\, x_i\in E\bigg\},
\]
where the infimum is taken over all countable covers of $E$ by balls
$B(x_i,r_i)$. We note here that if $X$ is a proper, i.e. boundedly
compact, metric space, then Hausdorff content is inner regular in the
following sense
\[
\H^s_\infty(E)= \sup\{\H^s_\infty(K):K \subset E,\, K \textrm{ compact}\}
\]
whenever $E \subset X$ is a Borel set. See Federer~\cite[Corollary
2.10.23]{Federer}. We shall also need the concept of the
\emph{Hausdorff measure of codimension $s$} of $E\subset X$ which we
define by applying the Carath\'eodory construction to the function
\[
h(B(x,r)) = \frac{\mu(B(x,r))}{r^s}.
\]
Above, we use the convention $h(B(x,0)) := h(\emptyset) = 0$. We thus
define the restricted Hausdorff content of codimension $s$ as follows
\[
\widetilde\H^s_R(E) = \inf\biggl\{\sum_{i=1}^\infty h(B(x_i,r_i)):
E\subset \bigcup_{i=1}^\infty B(x_i,r_i),\, x_i\in E,\, r_i \leq
R\biggr\},
\]
where $0<R<\infty$. When $R=\infty$, we have the corresponding
Hausdorff content of $E$ and denote it by
$\widetilde\H^s_\infty(E)$. Finally, the Hausdorff measure of
codimension $s$ is defined as
\[
\widetilde\H^s(E) = \lim_{R\to 0}\widetilde\H^s_R(E).
\]
We remark that if the measure $\mu$ is $Q$-regular, i.e., $\mu(B(x,r))
\approx r^Q$, for some $Q\geq 1$, $\widetilde\H^s(E) \approx
\H^{Q-s}(E)$. Let us mention that the lower mass bound
\eqref{measuregrowth} for the measure $\mu$ implies that $\H^Q$ is
absolutely continuous with respect to $\mu$ and that $\H^{Q-s}(E)\leq C
\widetilde\H^s(E)$.

The \emph{upper $s$-density} of a finite Borel regular measure $\zeta$
at $x$ is defined by
\[
\Theta^\ast_s (\nu,x) = \limsup_{r\to
  0\limplus}\frac{\zeta(B(x,r))}{\omega_sr^s },
\]
where $\omega_s$ is the Lebesgue measure of the unit ball in $\R^s$
when $s$ is a positive integer, and $\omega_s=
\Gamma(1/2)^s/\Gamma(s/2+1)$ otherwise. We record that if for all $x$ in a
Borel set $E\subset X$, $\Theta^\ast_s(\zeta,x)\geq \alpha$,
$0<\alpha<\infty$, then
\[
\zeta \geq \alpha C\H^s\rtangle E,
\]
where the positive constant $C$ depends only on $s$. On the other
hand, if $\Theta^\ast_s(\zeta,x)\leq\alpha$ we obtain
\[
\zeta\rtangle E\leq \alpha C\H^s\rtangle E,
\]
where a positive constant $C$ depends only on $s$. See
Federer~\cite[2.10.19]{Federer}. 

Recall that the following general covering theorem is valid in our
setting. From a given family of balls $\mathcal{B}$ with $\sup\{\diam
B:\ B\in\mathcal{B}\}<\infty$ covering a set $E\subset X$ we can
select a pairwise disjoint subfamily $\mathcal{B}'$ of balls such that
\[
E \subset \medcup_{B\in\mathcal{B}'} 5B,
\]
see \cite[Corollary 2.8.5]{Federer}. If $X$ is separable, then
$\mathcal{B}'$ is countable and $\mathcal{B}'= \{B_i\}_{i\geq 1}$.

In this note, a \emph{curve} $\gamma$ in $X$ is a continuous mapping
from a compact interval $[0,L]$ to $X$. We recall that each curve can
be parametrized by 1-Lipschitz map $\tilde\gamma:[0,L]\to X$. A
nonnegative Borel function $g$ on $X$ is an \emph{upper gradient} of
a function $f:X\to Y$ if for all rectifiable
curves $\gamma$, we have
\begin{equation} \label{ug-cond}
\tilde d(f(\gamma(L)),f(\gamma(0)))\leq \int_\gamma g\,ds.
\end{equation}
See Cheeger~\cite{Cheeger} and Shanmugalingam~\cite{Sh-rev} for a
discussion on upper gradients. If $g$ is a nonnegative measurable
function on $X$ and if (\ref{ug-cond}) holds for \p-almost every
curve, \(p\geq1\), then $g$ is a \emph{weak upper gradient} of $f$. By
saying that (\ref{ug-cond}) holds for \p-almost every curve we mean
that it fails only for a curve family with zero \p-modulus (see, e.g.,
\cite{Sh-rev}). If $u$ has an upper gradient in $L^p(X)$, then it has
a \emph{minimal weak upper gradient} $g_f \in L^p(X)$ in the sense
that for every weak upper gradient $g \in L^p(X)$ of $f$, $g_f \le g$
$\mu$-almost everywhere (a.e.), see Corollary~3.7 in
Shanmugalingam~\cite{Sh-harm}. While the results in \cite{Sh-rev} and
\cite{Sh-harm} are formulated for real-valued functions and their
upper gradients, they are applicable for metric space valued functions
and their upper gradients; the proofs of these results require only
the manipulation of upper gradients, which are always real-valued.

We define Sobolev spaces on metric spaces following
Shanmugalingam~\cite{Sh-rev}. Let $\Om\subseteq X$ be nonempty and
open. Whenever $u\in L^p(\Om)$ and \(p\ge1\), let
\begin{equation} \label{lpnorm}
        \|u\|_{\Np(\Om)} := \|u\|_{1,p} := \biggl( \int_\Om |u|^p \, \dmu
                + \int_\Om g_u^p \, \dmu \biggr)^{1/p}.
\end{equation}
The \emph{Newtonian space} on $\Om$ is the quotient space
$$
        \Np (\Om) = \{u: \|u\|_{\Np(\Om)} <\infty \}/{\sim},
$$
where $u \sim v$ if and only if $\|u-v\|_{\Np(\Om)}=0$. The space
$\Np(\Om)$ is a Banach space and a lattice. If $\Om \subset \R^n$ is
open, then $\Np(\Om) = W^{1,p}(\Om)$ as Banach spaces.  For these and
other properties of Newtonian spaces we refer to \cite{Sh-rev}. The
class $\Np(\Om;\R^m)$ consists of those mappings $u:\Om\to\R^m$ whose
component functions each belong to $\Np(\Om)=\Np(\Om;\R)$.
Qualitative properties like Lebesgue points, density of Lipschitz
functions, quasicontinuity, etc.  may be investigated componentwise.

A function belongs to the \emph{local Newtonian space}
$N\loc^{1,p}(\Om)$ if $u\in N^{1,p}(V)$ for all bounded open sets $V$
with $\bar V\subset\Om$, the latter space being defined by considering
$V$ as a metric space with the metric $d$ and the measure $\mu$
restricted to it.

Newtonian spaces share many properties of the classical Sobolev
spaces. For example, if $u,v \in N^{1,p}\loc(\Om)$, then
$g_u=g_v$ $\mu$-a.e. in $\{x \in \Omega : u(x)=v(x)\}$, furthermore,
$g_{\min\{u,c\}}=g_u \chi_{\{u \neq c\}}$ for $c \in \R$.

We shall also need a \emph{Newtonian space with zero boundary values}. For a
measurable set $E\subset \Om$, let
\[
N_0^{1,p}(E) = \{f|_E: f\in N^{1,p}(\Om) \textrm{ and } f= 0 \textrm{ on
} \Om\setminus E\}.
\]
This space equipped with the norm inherited from $N^{1,p}(\Om)$ is a
Banach space.

We say that $X$ supports a \emph{weak $(1,p)$-Poincar\'e inequality}
if there exist constants $C>0$ and $\tau \ge 1$ such that for all
balls $B(z,r) \subset X$, all measurable functions $f$ on $X$ and for
all weak upper gradients $g_f$ of $f$,

\begin{equation} \label{PI-ineq}
        \vint_{B(z,r)} |f-f_{B(z,r)}| \,\dmu
        \le Cr \Big( \vint_{B(z,\tau r)} g_f^{p} \,\dmu \Big)^{1/p},
\end{equation}
where $f_{B(z,r)} :=\vint_{B(z,r)} f \, d\mu :=\int_{B(z,r)} f
\,\dmu/\mu(B(z,r))$.

It is well known that the embedding $\Np(X) \rightarrow L^{p}(X)$ is
not surjective if and only if there exists a curve family in $X$ with
a positive \p-modulus. Moreover, the validity of a Poincar\'e
inequality can sometimes be stated in terms of \p-modulus. More
precisely, to require that \eqref{PI-ineq} holds in $X$ is to require
that the \p-modulus of curves between every pair of distinct points of
the space is sufficiently large, see Theorem 2 in Keith~\cite{Keith}.

It is noteworthy that by a result of Keith and Zhong~\cite{KeZho} in a
complete metric space equipped with a doubling measure and supporting
a weak $(1,p)$-Poincar\'e inequality there exists $\eps_0>0$ such that
the space admits a weak $(1,p')$-Poincar\'e inequality for each
$p'>p-\eps_0$.

The following Luzin-type approximation theorem shall be of use later
in the paper. We refer to Shanmugalingam~\cite[Theorem 4.1]{Sh-rev}
for the proof which, in turn, is a modification of an idea due to
S. Semmes. See also Haj\l asz~\cite[Theorem 5]{Hajlasz}.

\medskip

\begin{thm} \label{thm:Luzin} Suppose $X$ satisfies (D) and (PI) for
  some $1<p<\infty$. Let $u\in N^{1,p}(X)$. Then for every
  $\eps>0$ there is a Lipschitz function $f_\eps: X\to
  \R$ such that
\[
\mu(\{x\in X : u(x)\neq f_\eps(x)\})< \eps
\]
and $\|u-f_\eps\|_{1,p}<\eps$.  In other words, with $F_{\eps} :=
\{x\in X : u(x)\neq f_\eps(x)\}$, we have $u|_{X \setminus F_{\eps}}$
is Lipschitz.
\end{thm}

\subsection*{Capacity} There are several equivalent definitions for
capacities, and the following are the ones we find most suitable for
our purposes. Let $1\leq p<\infty$ and $\Om\subset X$ bounded.
\begin{itemize}

\item The variational \p-capacity of a set $E\subset X$ is the number
\[
\capc_p(E) = \inf\|g_u\|_{L^p(X)}^p,
\]
where the infimum is taken over all $u\in\Np(X)$ such that $u\geq 1$
on $E$; recall that $g_u$ is the minimal \p-weak upper gradient of
$u$.

\item The relative \p-capacity of $E\subset \Om$ is
the number
\[
\Capc_p(E,\Om) = \inf\|g_u\|_{L^p(\Om)}^p,
\]
where the infimum is taken over all $u\in N_0^{1,p}(\Om)$ such that $u\geq 1$
on $E$.

\item The Sobolev \p-capacity of $E\subset X$ is the number
\[
\Cp(E) = \inf\|u\|_{N^{1,p}(X)}^p,
\]
where the infimum is taken over all $u\in N^{1,p}(X)$ such that $u\geq 1$
on $E$.
\end{itemize}

Observe that if $\mu(X)<\infty$ the constant function will do as a
test function, thus all sets are of zero variational
\p-capacity. However, this is not true for the relative \p-capacity
whenever $X\setminus \Om$ is ``large'', say, $\Cp(X\setminus\Om)>0$.

Under our assumptions, these capacities enjoy the standard properties
of capacities. For instance, when $p>1$ they are Choquet
capacities, i.e., the capacity of a Borel set can be obtained by
approximating with compact sets from inside and open sets from
outside. It is noteworthy, however, that the Choquet property fails
for $p=1$ in the general metric setting. This does not cause any
problems for us as we mainly deal with compact sets in this note. In a
recent paper by Kinnunen--Hakkarainen~\cite{HaKi} the BV-capacity was
proved to be a Choquet capacity. See, e.g.,
Kinnunen--Martio~\cite{KiMa96}, \cite{KiMaNov} for a discussion on
capacities on metric spaces.

The Sobolev capacity is the correct gauge for distinguishing between
Newtonian functions: if $u\in N^{1,p}(X)$, then $u\sim v$ if and only
if $u=v$ \p-quasieverywhere, i.e., outside a set of zero Sobolev
\p-capacity. Moreover, by Shanmugalingam~\cite{Sh-rev} if $u,v\in
N^{1,p}(X)$ and $u=v$ $\mu$-a.e., then $u\sim v$. A function $u\in
N^{1,p}(X)$ is said to be \emph{quasicontinuous}, if there exists an
open set $G\subset X$ with arbitrarily small Sobolev \p-capacity such
that the restriction of $u$ to $X\setminus G$ is continuous. A mapping
in $N^{1,p}(X;\R^m)$ is said to be quasicontinuous if each of its
component functions is quasicontinuous. Recall that \emph{all}
functions in $N^{1,p}(X)$ are quasicontinuous, see Bj\"orn et
al.~\cite{BBS}. Since Newtonian functions have Lebesgue points outside
a set of zero Sobolev capacity, in what follows we may assume that
every Newtonian function is precisely represented.

\section{Graphs of Newtonian functions: Luzin's condition}
\label{sect:graph}

Let $Q>0$. Recall that a mapping $f:X\to Y$ is said to satisfy
\emph{Luzin's condition $(N_Q)$} if $\H^Q(f(E)) = 0$ whenever
$E\subset X$ satisfies $\mu(E) = 0$. By way of motivation, the
validity of Luzin's condition implies certain change of variable
formulas, thus it is of independent interest in analysis.

Let $E\subset X$. We denote by $\bar{f}:X \to X\times Y$ the
\emph{graph mapping of $f$}
\[
\bar{f}(x) = (x,f(x)), \quad x\in X,
\]
and $\GG_f(E)$ is the \emph{graph} of $f$ over $E$ defined by
\[
\GG_f(E) = \{(x,f(x)):\ x\in E\} \subset X\times Y.
\]
It is well known that if the mapping $f$ is Borel measurable, then the
graph $\GG_f(X)$ is Borel measurable as well, see, e.g., \cite[Lemma
18]{HajlaszProc}. We, furthermore, denote by $\pr_X:X\times Y \to X$
the projection $\pr_X(x,y)=x$, and by $\pr_Y:X\times Y \to Y$ the
projection $\pr_Y(x,y)=y$. Observe that $\Lip(\pr_X) = \Lip(\pr_Y) =
1$. Also it is well-known that if $f:X\to Y$ is continuous, then
$\GG_f(X)$ is homeomorphic to $X$.

\medskip

\begin{lemma} \label{measurability} Let $f:X\to \R^m$, $m\geq 1$, be
  measurable. Then $\pr_X(\GG_f(X)\cap E)$ is measurable for
  every Borel measurable subset $E\subset X\times \R^m$.
\end{lemma}

\begin{proof}
  Let $f^\ast$ and $f_\ast$ be Borel measurable representatives of
  $f$; Borel regularity of the measure $\mu$ implies that if $f$ is
  measurable, then there exist Borel measurable functions
  $f_\ast,f^\ast$ such that $f_\ast\leq f \leq f^\ast$ and
  $f_\ast(x)=f^\ast(x)$ for $\mu$-a.e. $x\in X$. Thus the graph
  $\GG_{f_\ast}(X)$ of $f_\ast$ and the graph $\GG_{f^\ast}(X)$ of
  $f^\ast$ are Borel subsets of $X\times \R^m$. Then
  Kuratowski~\cite[Theorem 2, p. 385]{Kuratowski} implies that the
  projections $\pr_X(\GG_{f_\ast}(X)\cap E)$ and
  $\pr_X(\GG_{f^\ast}(X)\cap E)$ are Borel measurable for every Borel
  measurable set $E \subset X\times \R^m$. Since $f_\ast$ and $f^\ast$
  agree up to a set of $\mu$-measure zero, so do sets
  $\pr_X(\GG_{f^*}(X)\cap E)$ and $\pr_X(\GG_{f_*}(X)\cap E)$, implying
  that $\pr_X(\GG_f(X)\cap E)$ is $\mu$-measurable.
\end{proof}

\medskip

We now state a general criterion for the condition $(N_Q)$ similar to
that of Rad\'o and Reichelderfer, see \cite[V.3.6]{RaRe} and
Mal\'y~\cite{Maly}.  In Euclidean spaces this result was obtained by
Mal\'y et al.~\cite{MaSwaZi}.

In what follows, we suppose that $1\leq m < Q$, where $m$ is
related to $\R^m$.

\medskip

\begin{thm} \label{thm:condN} Suppose $X$ satisfies condition (D) and
  the lower mass bound \eqref{measuregrowth} is satisfied. Let $f:X^Q
  \to \R^m$ be a measurable function. Denote $$\Xi_{z,r} = \GG_f(X^Q)\cap
  B(z,r),$$ where $z\in X^Q\times\R^m$ and $0<r<\diam(X^Q)$. Suppose that
  there exists a weight $\Phi \in L^1\loc(X^Q)$ such that
  \begin{equation} \label{eq:cond} \H_\infty^{Q-m}(\pr_X(\Xi_{z,r}))\leq
    \frac1{\diam(\Xi_{z,r})^m}\int_{\pr_X(\Xi_{z,4r})}\Phi\,d\mu
\end{equation}
for all $z\in X^Q\times \R^m$ and all $0<r<\diam(X^Q)/4$. Then there
exists a positive constant $C<\infty$, depending on $C_\mu$ and $m$,
such that
\begin{equation} \label{quantitativecondN}
\H^Q(\bar{f}(E)) \leq C\int_E\Phi\,d\mu
\end{equation}
for each Borel measurable set $E \subset X^Q$. In particular,
$\bar{f}$ satisfies Luzin's condition $(N_Q)$.
\end{thm}

\begin{proof}
  Define a set function $\sigma$ on the Cartesian product $X^Q\times
  \R^m$ by
\[
\sigma(E) = \int_{\pr_X(\GG_f(X^Q) \cap E)}\Phi\,d\mu, \quad E\subset
X^Q\times\R^m.
\]
By a Vitali-type covering theorem there is a pairwise disjoint
countable subfamily of balls $\{B_i\}:=\{B(x_i,r_i)\}$ such that we
may cover $\pr_X(\Xi_{z,r})$ as follows
\[
\pr_X(\Xi_{z,r}) \subset \bigcup_i B(x_i,5r_i) =:\bigcup_i 5B_i.
\]
For each $i$ let $M_i$ denote the greatest integer satisfying
\begin{equation*}
(M_i-1)r_i < \diam(\Xi_{z,r}).
\end{equation*}
Since $\Xi_{z,r} \cap \pr_X^{-1}(5B_i)$ is bounded in $X^Q\times \R^m$, it
can be contained in a large enough cylinder of the form $B(x_i,5r_i)
\times \mathcal{R}_i$, where $\mathcal{R}_i$ is a cube in $\R^m$ with
side-length $\diam(\Xi_{z,r})$. Since $M_ir_i \geq \diam \Xi_{z,r}$,
$\mathcal{R}_i$ may be covered by $M_i^m$ cubes $\{\mathcal{R}_i^j\}$
with side $r_i$. We hence obtain
\begin{align*}
  \H_\infty^Q & (\Xi_{z,r} \cap \pr_X^{-1}(5B_i)) \leq CM_i^mr_i^Q \leq C(M_ir_i)^mr_i^{Q-m} \\
  & \leq C(\diam(\Xi_{z,r}) + r_i)^m\mu(5B_i)(5r_i)^{-m}.
\end{align*}
As $r_i\approx \diam(5B_i) \leq \diam\pr_X(\Xi_{z,r}) \leq \diam(\Xi_{z,r})$ summing over $i$
shows that
\[
\H_\infty^Q(\Xi_{z,r}) \leq
C\diam(\Xi_{z,r})^m\sum_{i=1}^\infty\frac{\mu(5B_i)}{(5r_i)^m}.
\]
Hence by taking the infimum over all coverings we have obtained the
following estimate
\[
\H_\infty^Q(\Xi_{z,r}) \leq
C\diam(\Xi_{z,r})^m\widetilde\H_\infty^m(\pr_X(\Xi_{z,r})),
\]
where the constant $C$ depends only on $C_\mu$ and $m$. Assumption
\eqref{eq:cond} together with this estimate gives for each $z\in
X\times \R^m$ and $0<r<\diam(X^Q)/4$
\begin{align} \label{eq:sigma}
  \H_\infty^Q(\Xi_{z,r}) & \leq C\diam(\Xi_{z,r})^m\widetilde\H_\infty^m(\pr_X(\Xi_{z,r})) \\
  & \leq C\int_{\pr_X(\Xi_{z,4r})}\Phi\,d\mu \leq
  C\sigma(B(z,4r)). \nonumber
\end{align}
Since for $\H^Q$-almost every $z \in \GG_f(X^Q)$, see
Federer~\cite[Lemma 10.1]{FedererTAMS},
\begin{equation}\label{upperdest}
\limsup_{r\to 0\limplus}
\frac{\H_\infty^Q (\Xi_{z,r})}{\omega_Qr^Q } \geq C,
\end{equation}
it follows from \eqref{eq:sigma} that
\[
\limsup_{r\to 0\limplus} \frac{\sigma(B(z,r))}{\omega_Qr^Q} \geq C
\]
for $\H^Q$-almost every $z \in \GG_f(X^Q)$. Lemma~\ref{measurability}
implies that $\sigma$ is a measure on the Borel sigma algebra of
$X^Q\times \R^m$, and it may be extended to a regular Borel outer
measure $\sigma^\ast$ on all of $X^Q\times\R^m$ in the usual way
\[
\sigma^\ast(A) :=\inf\{\sigma(E):\ A\subset E,\ E \textrm{ is a Borel set}\}.
\]
Since $\Phi \in L^1\loc(X^Q)$ it follows that $\sigma^\ast$ is a
Radon measure on $X^Q\times \R^m$. Therefore, by \eqref{upperdest}
\[
\H^Q (E) \leq C\sigma^\ast(E)
\]
for all $E\subset \GG_f(X^Q)$. Finally, given a $\mu$ measurable set
$E\subset X^Q$, choose a Borel set $G$ with $E\subset G$. Then
$\bar{f}(E)\subset G\times \R^m$, $G\times \R^m$ is a Borel set, and
\[
\H^Q(\bar{f}(E)) \leq C\sigma^\ast(\bar{f}(E)) \leq
C\sigma(G\times \R^m) = C\int_G\Phi\,d\mu.
\]
The proof is completed by taking the infimum over all such $G$.  If
$E\subset X^Q$ such that $\mu(E) = 0$ then it readily follows that $\H^Q
(\overline{f}(E)) = 0$. This completes the proof.
\end{proof}

\medskip

In \eqref{eq:cond} we may replace the Hausdorff content
$\H_\infty^{Q-m}(\pr_X(\Xi_{z,r}))$ with an inequality involving
$\widetilde\H_\infty^m(\pr_X(\Xi_{z,r}))$ on the left hand side.

We shall show, as an application of Theorem \ref{thm:condN}, that the
graph mapping of a Newtonian function satisfies a version of Luzin's
condition $(N_Q)$. We start with a few auxiliary estimates. We shall
need the following relation between the \p-capactity and the Hausdorff
content when $p\geq 1$. For the proof of the next lemma the reader
should consult Costea~\cite[Thoerem 4.4]{Costea} and Kinnunen et
al. in \cite[Theorem 3.5]{KiKoShTu} for the case (I) and (II),
respectively.

\medskip

\begin{lemma} \label{lemma:capcontent} Suppose $X$ satisfies
  conditions (D) and (PI), and the lower mass bound
  \eqref{measuregrowth} is satisfied.
\begin{itemize}

\item[(I)] Let $1<p\leq Q$ and $E\subset X$ and suppose $Q-p<t\leq
  Q$. Then
\[
\H_\infty^t (E\cap B(x,r)) \leq Cr^{t-Q+p}\Capc_{p}(E\cap B(x,r),B(x,2r)),
\]
where $x\in X$, $r>0$, and $C$ depends on $C_\mu$, $p$, $t$, and the
constants in the weak $(1,p)$-Poincar\'e inequality.

\item[(II)] Let $p=1$ and $E\subset X$ compact. Then
\[
\widetilde\H_\infty^1(E)\leq C\capc_1(E),
\]
where the constant $C$ depends only on the doubling constant $C_\mu$
and the constants in the weak $(1,1)$-Poincar\'e inequality.
\end{itemize}
\end{lemma}

\medskip

\begin{remark}
  If $u\in N^{1,p}_0(B(x,2r);\R^m)$ such that $u \geq 1$ on $E\cap
  B(x,r)$, $g_u$ is a minimal $p$-weak upper gradient of $u$, and $m$,
  where $1\leq m < \min\{p,Q\}$, we obtain
\[
\H_\infty^{Q-m}(E\cap B(x,r)) \leq Cr^{p-m}\int_{B(x,2r)} g_u^p\,
d\mu,
\]
where the constant $C$ is as in Lemma~\ref{lemma:capcontent} (I).

If $u\in N^{1,1}(X;\R)$ such that $u\geq 1$ on $E$ and $g_u$ is a
minimal 1-weak upper gradient of $u$, Lemma~\ref{lemma:capcontent}
(II) implies that
\[
\widetilde\H_\infty^1(E)\leq C\int_X g_u\, d\mu,
\]
where the constant $C$ is from Lemma~\ref{lemma:capcontent} (II).
\end{remark}

\medskip

The preceding estimates imply the following. Observe also that the
graph mapping is always one-to-one.

\medskip

\begin{thm} \label{thm:ucondN} Suppose that $X$ satisfies conditions
  (D) and (PI) with some $1\leq p\leq Q$, and the lower mass bound
  \eqref{measuregrowth} is satisfied.  Let $u \in N^{1,p}(X^Q;\R^m)$,
  where either $p > m$ or $p\geq m=1$. Then the graph mapping
  $\overline{u}$ satisfies Luzin's condition $(N_Q)$.
\end{thm}

\medskip

The assumption that $p>m$ or $p\geq m=1$ is necessary already in the
Euclidean case. We refer to a discussion in Mal\'y et
al.~\cite{MaSwaZi}.

\begin{proof}[Proof of Theorem~\ref{thm:ucondN}]
  It is sufficient to verify the hypothesis of Theorem~\ref{thm:condN}
  with some locally integrable function $\Phi$ on $X^Q$.

  Assume first $p>m$ and, to this end, fix a point $z=(\tilde x,\tilde
  y) \in X^Q\times \R^m$ and $r>0$. We observe the following
\[
\Xi_{z,r} = \GG_u(X^Q)\cap B(z,r) \subset (\GG_u(X^Q) \cap (B_X(\tilde
x,r)\times B(\tilde y,r))).
\]
Hence we have that
\[
\pr_X(\Xi_{z,r}) \subset (B_X(\tilde x,r)\cap u^{-1}(B(\tilde y,r))),
\]
moreover $u(x) \in B(\tilde y,r)$ for $\mu$-a.e. $x\in B_X(\tilde x,r)\cap
u^{-1}(B(\tilde y,r))$. Let us define the function $v:X^Q\to \R$ by
\[
v(x) = \max\left\{2-\frac{|u(x)-u(\tilde x)|}{r},0\right\},
\]
and consider an open subset $O\subset X^Q$ such that $\{x\in X^Q:
v(x)>0\}\subset O$. Then $(g_u/r)\chi_O$ is a \p-weak upper gradient of
$v$ \cite[Lemma 4.3]{Sh-rev}, where $g_u$ is a minimal \p-weak upper
gradient of $u$. Let $\eta:X^Q \to \R$ be a Lipschitz cut-off function
so that $\eta = 1$ on $B_X(\tilde x,r)$, $\eta = 0$ in $X^Q\setminus
B_X(\tilde x,2r)$, $0\leq\eta\leq 1$, and $g_\eta \leq 2/r$. Then
$v\eta \geq 1$ on $B_X(\tilde x,r)\cap u^{-1}(B(\tilde y,r))$, and
$v\eta \in N_0^{1,p}(B(\tilde x,2r))$. Moreover, the product rule for
upper gradients gives us the following $g_{v\eta}\leq g_v + 2v/r$
$\mu$-a.e. Thus $v\eta$ is admissible for the relative \p-capacity and
Lemma~\ref{lemma:capcontent} (I) implies that
\begin{align*}
  \H_\infty^{Q-m} & (\pr_X(\Xi_{z,r})) \leq \H_\infty^{Q-m}(B_X(\tilde x,r)\cap u^{-1}(B(\tilde y,r))) \\
  & \leq Cr^{p-m}\int_{B_X(\tilde x,2r)\cap O}g_{v\eta}^{p}\, d\mu \\
  & \leq Cr^{p-m}\int_{B_X(\tilde x,2r)\cap O}
  \left(\frac{v^{p}}{r^p} + g_v^{p}\right)\, d\mu  \\
  & \leq Cr^{-m}\int_{B_X(\tilde x,2r)\cap u^{-1}(B(\tilde
    y,2r))}(1+g_u^{p})\, d\mu.
\end{align*}
Since
\[
B_X(\tilde x,2r)\cap u^{-1}(B(\tilde y,2r)) \subset \pr_X(\Xi_{z,4r}),
\]
above reasoning gives us that
\[
\H_\infty^{Q-m}(\pr_X(\Xi_{z,r})) \leq \frac{C}{r^m}\int_{\pr_X(\Xi_{z,4r})}(1+g_u^{p})\, d\mu.
\]
This verifies the assumptions of Theorem~\ref{thm:condN} with $\Phi =
C(1+g_u^{p})$, and thus concludes the proof when $p>m$.  The case
$p\geq m=1$ is dealt with by a similar argument together with the
estimate in Lemma~\ref{lemma:capcontent} (II).
\end{proof}

\section{Aspects of area formulas for Newtonian functions}
\label{sect:area}

In this section we shall prove versions of the area formula for
Newtonian functions. In the metric measure space setting these
formulas have been studied previously by
Ambrosio--Kirchheim~\cite{AmKi}, Magnani~\cite{MagnaniIndiana,
  MagnaniX}, and Mal\'y~\cite{MalyX, MalyCoarea}, to name but a
few. In particular, in \cite{MalyX} coarea properties and coarea
formula, which is considered as dual to the area formula, are
thoroughly studied in metric spaces. We also refer to Haj\l
asz~\cite{HajlaszProc} for a very nice discussion on the topic in
Euclidean spaces.

We define the generalized Jacobian of a continuous map $f:X\to Y$ at
$x$ as follows
\[
\mathcal{J}f(x) :=\limsup_{r\to 0}\frac{\nu(f(B(x,r)))}{\mu(B(x,r))},
\]
where, we recall, $\nu$ measures $Y$. It follows from
\cite[2.2.13]{Federer} applied to the pull-back measure $\nu_f(E)
:=\nu(f(E))$, that $f(E)$ is $\nu$-measurable for every Borel set
$E\subset X$. Moreover, for $\mu$-a.e. $x$, the generalized Jacobian
$\mathcal{J}f(x)$ is finite, see Federer~\cite[2.9]{Federer}. It is
also easy to see that if $g:X\to Y$ is another continuous map such
that $g=f$ on an open subset $A\subset X$, then
$\mathcal{J}f(x)=\mathcal{J}g(x)$ for $\mu$-a.e. $x\in A$.

An alternative, but maybe less tractable, way to define a generalized
Jacobian of $f$ at $x$ could be as follows. Set
\[
\mathcal{\widetilde J}f(x):=\limsup_{r\to
  0}\frac{f^*\nu(B(x,r))}{\mu(B(x,r))},
\]
where $f^*\nu$ is a measure which results by Carath\'eodory's
construction from $\zeta(A)=\nu(f(A))$, $A\subset X$, on the family of
all Borel subsets of $X$, see \cite[2.10.1]{Federer}. Hence if $A$ is a Borel
subset of $X$, then 
\[
f^*\nu(A) = \sup\left\{\sum_{B\in \mathcal{H}}\zeta(B): \mathcal{H}
  \textrm{ is a Borel partition of } A\right\}
\]
cf. \cite[Theorem 2.10.8]{Federer}; for any Borel set $A\subset X$ the
following identity will be satisfied \cite[Theorem 2.10.10]{Federer}
\[
f^*\nu(A) = \int_Y N(f|_A, y)\, d\nu(y),
\]
where the \emph{multiplicity function} of $f$ relative to a subset $A$
is written as $N(f|_A, y) = \# (A\cap f^{-1}(y))$ for each $y\in Y$.

To compare these two notions, we have that
\[
\mathcal{J}f(x)= \mathcal{\widetilde J}f(x)=\mathcal{J}f|_D(x) 
\]
for $\mu$-a.e. $x\in D$, where $D\subset X$ is closed and $f|_D$ is
assumed to be one-to-one. Here we denote
\[
\mathcal{J}f|_D(x):=\limsup_{r\to 0}\frac{\nu(f(B(x,r)\cap
  D))}{\mu(B(x,r))}.
\]
Let us clarify this. Clearly, $\mathcal{J}f|_D(x)\leq \mathcal{J}f(x)
\leq \mathcal{\widetilde J}f(x)$ for $\mu$-a.e. $x\in D$. On the other
hand, since $f$ is one-to-one on $D$ we have that
$\zeta(A):=\nu(f(A))$ is, in fact, a measure on $D$, and that
$\zeta(A) = f^*\nu(A)$ for every Borel subset of $D$. Thus we obtain
as in Magnani~\cite[proof of Theorem 2]{MagnaniX} for every (density
point) $x\in D$
\begin{align*}
  \mathcal{\widetilde J}f(x) & \leq \limsup_{r\to
    0}\frac{\nu(f(B(x,r)\cap D))}{\mu(B(x,r))}
  + \limsup_{r\to 0}\frac{f^*\nu(B(x,r)\setminus D)}{\mu(B(x,r))} \\
  & = \mathcal{J}f|_D(x),
\end{align*}
where the last equality follows form \cite[Corollary 2.9.9]{Federer}
applied to $\mathcal{\widetilde J}f(x)\chi_D$, where $\chi_D$ is the
characteristic function of the set $D$.

Magnani~\cite{MagnaniX} has recently presented a unified approach to
the area formula for merely continuous mappings between metric spaces,
and thus without any notion of differentiability. We remark that in
the present paper a function in $N\loc^{1,p}(X^Q;\R^m)$ although
having some ``differentiability'' properties, need not to be even
continuous as all Newtonian functions are, a priori, only
quasicontinuous. Let us state the following area formula.

\medskip

\begin{theorem} \label{thm:area}
  Suppose $X$ satisfies conditions (D) and (PI) with some $1\leq p\leq
  Q$, and the lower mass bound \eqref{measuregrowth} is satisfied. Let
  $u \in N\loc^{1,p}(X^Q;\R^m)$, where $p > m$ or $p \geq m=1$.  Then
  the following area formula is valid
  \begin{equation} \label{eq:Area1} \H^Q(\bar u(A)) = \int_A
    \mathcal{J}\bar u(x)\, d\mu(x),
\end{equation}
whenever $A\subset X$ is $\mu$-measurable. 
\end{theorem}

\begin{proof}
  By Theorem~\ref{thm:ucondN} the graph mapping $\bar u$ satisfies
  Luzin's condition ($N_Q$) and is, moreover, one-to-one on $X$. Thus
  the pull-back measure $\H^Q(\bar u (A))$, $A\subset X^Q$ arbitrary
  $\mu$-measurable subset, is absolute continuous with respect to the
  doubling measure $\mu$.

  Let $\{f_i\}_{i\geq 1}$, $f_i: X^Q \to \R^m$, be a sequence of
  Lipschitz functions and $E_1\subset E_2 \subset \ldots \subset X^Q$
  associated closed sets such that $u_i:= u|_{E_i} = f_i|_{E_i}$ and
  $\mu(X^Q\setminus \bigcup_i E_i) = 0$. The existence of such sets
  and functions follows from Theorem~\ref{thm:Luzin}. Then the
  following identity is valid by the area formula obtained in
  \cite{MagnaniX}
  \begin{equation} \label{eq:area} \int_{E_i}\mathcal{J}\bar{f}_i(x)\,
    d\mu(x) = \H^Q(\bar f_i(E_i)).
\end{equation}
Since $u_i(x)=f_i(x)$ for $x\in E_i$, $E_i$ closed, it follows that
$\mathcal{J}\bar u_i(x) = \mathcal{J}\bar f_i(x)$ for $\mu$-a.e. $x\in
E_i$. The equality \eqref{eq:area} remains true for measurable
$A\subset E_\infty$, where $E_\infty=\bigcup_{i=1}^\infty E_i$, and moreover,
\eqref{eq:area} will also be valid whenever $\mu(A)=0$. Thus
\eqref{eq:Area1} holds for all $\mu$-measurable set $A\subset X^Q$.
\end{proof}

\medskip

Let us discuss an alternative formulation of the area formula which
can be obtained by using Theorem~2 in Magnani~\cite{MagnaniX}. Assume
$X$ satisfies conditions (D) and (PI) with some $1\leq p<\infty $, and
assume further that there exist disjoint $\mu$-measurable sets
$\{A_j\}_{j\geq 1}$ such that they occupy $\mu$-a.e. of $X$, i.e.
$\mu(X\setminus\bigcup_j A_j)=0$. Let $u \in N\loc^{1,p}(X^Q;\R^N)$,
where $Q\leq N$. Assume further that $u$ satisfies Luzin's condition
($N_Q)$ and $u|_{A_j}$ is one-to-one for each $i=1,2,\ldots\,$. Then
the following area formula is valid
\[
\int_A \theta(x)\mathcal{J}u(x)\, d\mu(x) =
\int_{\R^N}\sum_{x\in u^{-1}(y)}\theta(x)\, d\H^N(y),
\]
whenever $A\subset X$ is $\mu$-measurable and $\theta:A\to[0,\infty]$
is a measurable function. In particular,
\[
\int_A \mathcal{J}u(x)\, d\mu(x) =
  \int_{\R^N} N(u|_A, y)\, d\H^N(y)
  \]
is valid whenever $A\subset X$ is $\mu$-measurable.

\section{Newtonian functions: absolute continuity, Rad\'o,
  Reichelderfer, and Mal\'y}
\label{sect:RR}

Absolutely continuous functions on the real line satisfy Luzin's
condition, are continuous, and differentiable almost everywhere. It is
well-known that these properties for the Sobolev class $W^{1,p}(\R^m)$
depend on $p$. For instance, functions in $W^{1,m}(\R^m)$ may be
nowhere differentiable and nowhere continuos whereas functions in
$W^{1,p}(\R^m),\, p>m$, have H\"older continuous representatives and are
differentiable almost everywhere. We consider Luzin's condition,
absolute continuity, and differentiability for the Banach space valued
Newtonian space $N^{1,p}(X^Q;\V)$, when $p\geq Q$, and thus extend
some related results studied in Heinonen et al.~\cite{HKST}. Here $\V
:= (\V,\|\cdot\|_\V)$ is an arbitrary Banach space of positive
dimension. We refer the reader to \cite{HKST} for a
detailed discussion on the Banach space valued Newtonian
functions. Suppose $X$ satisfies conditions (D) and (PI) with some
$1\leq p<\infty$; the following is known:

\medskip

\begin{itemize}

\item Let $p>Q$. In this case each function $u\in N^{1,p}(X^Q;\R)$ is
  locally $(1-Q/p)$-H\"older continuous
  (Shanmugalingam~\cite{Sh-rev}), moreover $u$ is differentiable
  $\mu$-a.e. with respect to the strong measurable differentiable
  structure (see Cheeger~\cite{Cheeger}). For the latter result we
  refer to Balogh et al.~\cite{BaRoZu}.

\item Let $p=Q$. Then every continuous pseudomonotone mapping in
  $N^{1,Q}\loc(X^Q; \V)$ satisfies Luzin's condition $(N_Q)$ (Heinonen
  et al.~\cite[Theorem 7.2]{HKST}).

\end{itemize}

\medskip

It would be interesting to generalize Calderon's differentiability
theorem to Banach space valued Newtonian functions.

Recall that following  Mal\'y--Martio~\cite{MaMa}, a map $f : X \to \V$
is \emph{pseudomonotone} if there exists a constant $C_M\geq 1$ and $r_M
> 0$ such that
\[
\diam(f(B(x,r))) \leq C_M\diam(f(\partial B(x,r)))
\]
for all $x\in X$ and all $0 < r < r_M$. Note that we denote $\partial
B(x,r) := \{y\in X : d(y,x)=r\}$.

Let $\Om$ be open such that $\overline{\Om} \subset X^Q$. We show next
that $u\in N^{1,p}(\Om;\V)$, $p\geq Q$, is absolutely continuous in
the following sense. Following Mal\'y~\cite{Maly} we say that a
mapping $f:\Om\to \V$ is \emph{$Q$-absolutely continuous} if for each
$\eps>0$ there exists $\delta=\delta(\eps)>0$ such that for every
pairwise disjoint finite family $\{B_i\}_{i=1}^\infty$ of (closed)
balls in $\Om$ we have that
\[
\sum_{i=1}^\infty\diam(f(B_i))^Q < \eps,
\]
whenever $\sum_{i=1}^\infty\mu(B_i) <\delta$. Furthermore, we say that
a mapping $f: X \to \V$ satisfies the \emph{$Q$-Rad\'o--Reichelderfer
  condition}, \emph{condition (RR)} for short, if there exists a
non-negative control function $\Phi_f\in L^1\loc(X)$ such that
\begin{equation} \label{RR}
\diam(f(B(x,r)))^Q \leq \int_{B(x,r)}\Phi_f\, d\mu
\end{equation}
for every ball $B(x,r)\subset X$ with $0<r<R$. A condition similar to
this was used by Rad\'o and Reichelderfer in \cite[V.3.6]{RaRe} as a
sufficient condition for the mappings with the condition (RR) to be
differentiable a.e. and to satisfy Luzin's condition, see also
Mal\'y~\cite{Maly}. A function $f$ is said to satisfy condition (RR)
weakly if \eqref{RR} holds true with a dilated ball $B(x,\alpha r)$,
$\alpha>1$, on the right-hand side of the equation.

It readily follows that condition (RR) implies (local) $Q$-absolute
continuity of $f$. Indeed, let $\eps>0$ and $\{B(x_i,r_{x_i})\}$,
$0<r_{x_i}<R$, a pairwise disjoint finite family of balls in $\Om$
such that $E = \bigcup_iB(x_i,r_{x_i})$, and $\mu(E)<\delta$. Then
condition (RR) and pairwise disjointness of $\{B(x_i,r_{x_i})\}$ imply
\[
\sum_i\diam(f(B(x_i,r_{x_i})))^Q \leq
\sum_i\int_{B(x_i,r_{x_i})}\Phi_f\, d\mu = \int_{E}\Phi_f\, d\mu <\eps.
\]
Local absolute continuity of a function follows even if the functions
satisfies condition (RR) weakly.

Condition (RR) also implies that the map $f$ has finite pointwise
Lipschitz constant almost everywhere, see
Wildrick--Z\"urcher~\cite[Proposition 3.4]{WiZu}. Combined with a
Stepanov-type differentiability theorem \cite{BaRoZu}, this has
implications for differentiability \cite{Cheeger}. We also refer to a
recent paper \cite{WiZu2}.

For the next proposition, we recall that the noncentered
Hardy--Littlewood maximal function restricted to $\Om$, denoted
$\mathcal{M}_\Om$, is defined for an integrable (real-valued) function $f$ on
$\Om$ by
\[
\mathcal{M}_\Om f(x) := \sup_{B}\avint_{B(x,r)}|f|\, d\mu,
\]
where the supremum is taken over all balls $B\subset\Om$ containing
$x$. Consider further the restrained noncentered maximal function
$\mathcal{M}_{\Om,R}$ in which the supremum is taken only over balls
in $\Om$ with radius less than $R$. Then $\mathcal{M}_\Om f =
\sup_{R>0}\mathcal{M}_{\Om,R}f$. It is standard also in the metric
space setting, we refer to Heinonen~\cite{heinonen}, that for $1<p\leq
\infty$ the operator $\mathcal{M}_\Om$ is bounded on $L^P$, i.e.,
there exists a constant $C$, depending on $C_\mu$ and $p$, such that
for all $f\in L^p$
\[
\|\mathcal{M}f\|_{L^p} \leq C\|f\|_{L^p}.
\]
We have the following generalization.

\medskip

\begin{prop} \label{AC} Suppose $X$ satisfies conditions (D) and (PI)
  with

\begin{itemize}

\item[(I)] $p=Q$. If $u\in N\loc^{1,Q}(X^Q; \V)$ is continuous and
  pseudomonotone, then $u$ satisfies condition (RR), and thus is
  (locally) $Q$-absolutely continuous.

\item[(II)] some $p>Q$. Then $u\in N\loc^{1,p}(X^Q; \V)$ satisfies
  condition (RR) weakly, and thus is (locally) $Q$-absolutely
  continuous.

\end{itemize}
\end{prop}

\begin{proof}
  Let $\Om \Subset X^Q$ be open, and fix
  $x\in\Om$.

  {\bf (I)}: Let $B(x,r_x)$, $0<r_x<\min\{r_D,r_M\}$, be a ball such
  that $B(x,12\tau r_x)\subset \Om$; $\tau\geq 1$ is the dilatation
  constant appearing in the Poincar\'e inequality. By a Sobolev
  embedding theorem Haj\l asz--Koskela~\cite[Theorem 7.1]{KoHa} there
  exists a constant $C$, depending on $C_\mu$ and the constants in the
  weak $(1,Q)$-Poincar\'e inequality, and a radius $r_x < r< 2r_x$
  such that
 \begin{equation} \label{Sobo}
 \|u(z)-u(y)\|_\V^p \leq Cd(z,y)^{p/Q}r_x^{p(1-1/Q)}\avint_{B(x,5\tau
   r_x)}g_u^p\, d\mu
\end{equation}
for each $z,y\in \Omega$ with $d(y,x)=r=d(z,x)$, where $p\in
(Q-\eps_0,Q)$. In fact, \cite[Theorem 7.1]{KoHa} is stated and proved
only for real-valued functions, but the argument is valid also when
the target is a Banach space as we may make use of the Lebesgue
differentation theorem for Banach space valued maps as in
\cite[Proposition 2.10]{HKST}. Since $u$ is pseudomonotone we obtain
from \eqref{Sobo}
\[
\diam(u(B(x,r_x)))^p \leq C_M^p\diam u(\partial
  B(x,r)))^p \leq Cr_x^p\avint_{B(x,5\tau r_x)} g_u^p\, d\mu,
\]
where $C$ depends on $C_\mu$, $C_M$, and the constants in the weak
$(1,Q)$-Poincar\'e inequality. For each $y\in B(x,r_x)$ we have
\[
\avint_{B(x,5\tau r_x)} g_u^p\, d\mu \leq \avint_{B(y,10\tau r_x)}
g_u^p\, d\mu \leq \mathcal{M}_{\Om,12\tau r_x}g_u^p(y).
\]
Compining the preceding two estimates and
integrating over $y\in B(x,r_x)$ we get
\[
\diam(u(B(x,r_x)))^p \leq Cr_x^p\avint_{B(x,r_x)}\mathcal{M}_{\Om,12\tau r_x}
g_u^p\, d\mu.
\]
Recall that $Q-\eps_0< p < Q$; we get
\begin{align*}
  \diam(u(B(x,r_x)))^p & \leq Cr_x^p\mu(B(x,r_x))^{-p/Q} \\
  & \qquad \left(\int_{B(x,r_x)}\left(\mathcal{M}_{\Om,12\tau r_x}
      g_u^p\right)^{Q/p}\,
    d\mu\right)^{p/Q} \\
  & \leq Cr_x^p\mu(B(x,r_x))^{-p/Q}\left(\int_{B(x,r_x)}g_u^Q\,
    d\mu\right)^{p/Q},
\end{align*}
which implies together with \eqref{localmassest} that
\[
\diam(u(B(x,r_x)))^Q \leq C\tilde C\int_{B(x,r_x)}g_u^Q\, d\mu,
\]
where $C$ depends on $C_\mu$, $C_M$, and the constants in the weak
$(1,Q)$-Poincar\'e inequality, and $\tilde C$ is from
\eqref{localmassest}. As $g_u^Q\in L\loc^1(X)$ this verifies the fact
that $u$ satisfies condition (RR), and thus is locally $Q$-absolutely
continuous.

{\bf (II)}: Let $B(x,r_x)$, $0<r_x< r_D$, be a ball such that
$B(x,5\tau r_x)\subset \Om$. Theorem 5.1 (3) in Haj\l
asz--Koskela~\cite[Theorem 5.1]{KoHa} implies that there exist a
constant $C$, depending on $C_\mu$, $p$, and the constants appearing
in the weak $(1,p)$-Poincar\'e inequality, such that
 \[
 \|u(z)-u(y)\|_\V \leq Cd(z,y)^{1-Q/p}r_x^{Q/p}\left(\avint_{B(x,5\tau
   r_x)}g_u^p\, d\mu\right)^{1/p}
\]
for all $z,y\in B(x,r_x)$. In fact, \cite[Theorem 5.1]{KoHa} is stated
and proved only for real-valued functions, but the argument is valied
also when the target is a Banach space. Young's inequality $ab\leq
a^p/p + b^{p'}/p'$ and \eqref{localmassest} imply
\begin{align*}
  \diam(u(B(x,r_x)))^Q & \leq
  \frac{Cr_x^Q}{\mu(B(x,r_x))^{Q/p}}\left(\int_{B(x,5\tau
      r_x)}g_u^p\, d\mu\right)^{Q/p} \\
  & \leq C\left(\tilde C^{-1} \mu(B(x,r_x))+\int_{B(x,5\tau
      r_x)}g_u^p\, d\mu\right) \\
  & \leq C\left(\int_{B(x,\alpha r_x)}\left(\tilde C^{-1} +
      g_u^p\right)\, d\mu\right).
\end{align*}
Hence $u$ satisfies condition (RR) weakly with $\alpha =5\tau$ and
with $\Phi_u=C(\tilde C^{-1}+g_u^p)$, $\tilde C$ is from
\eqref{localmassest}.
\end{proof}

\medskip

The fact that a continuous pseudomonotone function $u\in
N\loc^{1,Q}(X^Q;\V)$ verifies Luzin's condition ($N_Q$) would easily
follow also from Proposition~\ref{AC} (I).


\begin{thebibliography}{99}

\bibitem{AmKi} \art{Ambrosio, L. \AND Kirchheim, B.}
  {Rectifiable sets in metric and Banach spaces}
  {Math. Ann.}{318}{2000}{527--555}

\bibitem{BaRoZu} \art{Balogh, Z. M., Rogovin, K. \AND Z\"urcher, T.}
  {The Stepanov differentiability theorem in metric measure spaces}
  {J. Geom. Anal.}{14}{2004}{405--422}

\bibitem{BBbook} \book{Bj\"orn, A. \AND Bj\"orn, J.}{Nonlinear
    Potential Theory on Metric Spaces}{EMS Tracts Math. 17, European
    Math. Soc., Z\"urich, 2011}

\bibitem{BBS} \art{Bj\"orn, A., Bj\"orn, J. \AND Shanmugalingam, N.}
  {Quasicontinuity of Newton--Sobolev functions and density of
    Lipschitz functions on metric spaces}
  {Houston J. Math.}{34}{2008}{1197--1211}

\bibitem{Cheeger} \art{Cheeger, J.}  {Differentiability of Lipschitz
    functions on metric measure spaces} {Geom. Funct. Anal.} {9}
  {1999} {428--517}

\bibitem{Costea} \art{Costea, S.}  {Sobolev capacity and Hausdorff
    measures in metric measure spaces}
  {Ann. Acad. Sci. Fenn. Math.}{34}{2009}{179--194}

\bibitem{FedererTAMS} \art{Federer, H.}  {The $(\varphi,k)$
    rectifiable subsets of $n$-space}
  {Trans. Amer. Soc.}{62}{1947}{114--192}

\bibitem{Federer} \book{Federer, H.}
        {Geometric Measure Theory}
        {Springer-Verlag, New York, 1969}

\bibitem{Hajlasz} \art{Haj\l asz, P.}  {Sobolev spaces on an
          arbitrary metric space} {Potential Anal.}{5}{1996}{403--415}

\bibitem{HajlaszProc} {\sc Haj\l asz, P.}, {Sobolev mappings,
          co-area formula and related topics}, {Proceedings on Analysis
          and Geometry, (Novosibirsk Akademgorodok, 1999), 227--254,
          Izdat. Ross. Akad. Nauk Sib. Otd. Inst. Mat., Novosibirsk,
          2000.}

\bibitem{KoHa} \art{Haj\l asz, P. \AND Koskela, P.}
    {Sobolev met Poincar\'e}
    {Mem. Amer. Math. Soc.}{145}{2000}{no. 688}

\bibitem{HaKi} \art{Hakkarainen, H. \AND Kinnunen, J} {The
      BV-capacity in metric spaces} {Manuscripta Math.}{132}{2010}{51--73}

\bibitem{heinonen} \book{Heinonen, J.}
        {Lectures on Analysis on Metric Spaces}
        {Springer-Verlag, New York, 2001}

\bibitem{HKST} \art{Heinonen, J., Koskela, P., Shanmugalingam, N.,
    Tyson, J.T.}  {Sobolev classes of Banach space-valued functions
    and quasiconformal mappings} {J. Anal. Math.}{85}{2001}{87--139}

\bibitem{Keith} \art{Keith, S.}
  {Modulus and the Poincar\'e inequality on metric measure spaces}
  {Math. Z.}{245}{2003}{255--292}

\bibitem{KeZho} \art{Keith, S. \AND Zhong, X.}
  {The Poincar\'e inequality is an open ended condition}
  {Ann. of Math.(2)}{167}{2008}{575--599}

\bibitem{KiKoShTu} \art{Kinnunen, J., Korte, R., Shanmugalingam, N.
    \AND Tuominen, H.}  {Lebesgue points and capacities via boxing
    inequality in metric spaces} {Indiana
    Univ. Math. J.}{57}{2008}{401--430}

\bibitem{KiMa96} \art{Kinnunen, J. \AND Martio, O.}  {The Sobolev
    capacity on metric spaces} {Ann. Acad. Sci. Fenn. Math.} {21}
  {1996} {367--382}

\bibitem{KiMaNov} {\sc Kinnunen, J. \AND Martio, O.}, {Choquet
    property for the Sobolev capacity in metric spaces}, in {\it
    Proceedings on Analysis and Geometry} (Novosibirsk, Akademgorodok,
  1999), pp. 285--290, Sobolev Institute Press, Novosibirsk, 2000.

\bibitem{Kuratowski} \book{Kuratowski, K.}  {Topology. Volume I}
  {Academic Press, 1966}

\bibitem{MagnaniIndiana} \art{Magnani, V.}
  {Area implies coarea}{Indiana Univ. Math. J.}{60}{2011}{77--100}

\bibitem{MagnaniX} \art{Magnani, V.}{An area formula in metric
    spaces}{Colloq. Math.}{124}{2011}{275--283}

\bibitem{Maly} \art{Mal\'y, J.}  {Absolutely continuous functions of
    several variables} {J. Math. Anal. Appl.}{231}{1999}{492--508}

\bibitem{MalyX} {\sc Mal\'y, J.}, {Coarea integration in metric
    spaces}, {Nonlinear Analysis, Function Spaces and Applications
    Vol. 7.  Proceedings of the Spring School held in Prague July
    17--22, 2002.  Eds. B. Opic and J. Rakosnik. Math. Inst. of the
    Academy of Sciences of the Czech Republic, Praha 2003,
    pp. 142--192}.

\bibitem{MalyCoarea} {\sc Mal\'y, J.}, {Coarea properties of Sobolev
    functions}, {Proc. Function Spaces, Differential Operators and
  Nonlinear Analysis, Teistungen, 2001, Birkh\"auser, Basel, 2003,
  pp. 371--381}.

\bibitem{MaMa} \art{Mal\'y, J. \AND Martio, O.}  {Lusin's condition
    (N) and mappings of the class $W\sp {1,n}$} {J. Reine
    Angew. Math.}{458}{1995}{19--36}

\bibitem{MaSwaZi} \art{Mal\'y, J., Swanson, D. \AND Ziemer, W. P.}
  {The coarea formula for Sobolev mappings}
  {Trans. Amer. Math. Soc.}{355}{2003}{477--492}


\bibitem{RaRe} \book{Rado, T. \AND Reichelderfer, P. V.}  {Continuous
    Transformations in Analysis} {Springer-Verlag,
    Berlin-Göttingen-Heidelberg, 1955}

\bibitem{Sh-rev} \art{Shanmugalingam, N.}  {Newtonian
      spaces\textup{:} An extension of Sobolev spaces to metric
      measure spaces} {Rev. Mat. Iberoamericana\/}{16}{2000}{243--279}

\bibitem{Sh-harm} \art{Shanmugalingam, N.}  {Harmonic functions on
      metric spaces} {Illinois J. Math.}{45}{2001}{1021--1050}

\bibitem{WiZu} {\sc Wildrick, K. \AND Z\"urcher, T.},
    {Mappings with an upper gradient in a Lorentz space}, {Preprint
      382, University of Jyv\"askyl\"a, 2009.}

\bibitem{WiZu2} {\sc Wildrick, K. \AND Z\"urcher, T.},
    {Sharp differentiability results for lip}, {Preprint, 2012.}

\end{thebibliography}
\end{document}